\documentclass{amsart}

\usepackage{latexsym,enumerate}
\usepackage{amsmath,amsthm,amsopn,amstext,amscd,amsfonts,amssymb}
\usepackage[ansinew]{inputenc}
\usepackage{verbatim}
\usepackage{epsfig} 
\usepackage{graphicx,psfrag} 

\newcommand{\CC}{\mathbb{C}}
\newcommand{\HH}{\mathbb{H}}

\newcommand{\RR}{\mathbb{R}}

\newcommand{\R}{\mathcal{R}}

\newcommand{\A}{\mathfrak{A}}
\newcommand\beqn{\begin{equation}}
\newcommand\neqn{\end{equation}}
\newcommand\beqnay{\begin{eqnarray}}
\newcommand\neqnay{\end{eqnarray}}

\newtheorem{teorema}{Theorem}[section]
\newtheorem{lema}[teorema]{Lemma}

\newtheorem{corolario}[teorema]{Corollary}

\DeclareMathOperator{\inter}{int}
\newcommand{\supp}[1]{\mbox{{\rm supp\/}}#1}

\newcommand{\spb}[1]{\smallskip}
\newcommand{\mpb}[1]{\medskip}
\newcommand{\bpb}[1]{\bigskip}

\newcommand{\p}{\partial}

\renewcommand{\a}{\alpha}
\renewcommand{\b}{\beta}
\newcommand{\al}{\alpha}
\newcommand{\be}{\beta}

\newcommand{\Ga}{\Gamma}
\newcommand{\de}{\delta}
\newcommand{\De}{\Delta}

\newcommand{\eps}{\varepsilon}

\renewcommand{\kappa}{\varkappa}

\renewcommand{\l}{\lambda}
\newcommand{\la}{\lambda}

\newcommand{\Om}{\Omega}

\newcommand{\wh}{\widehat}
\newcommand{\wt}{\widetilde}

\begin{document}

\title[
COMPUTATION OF CONFORMAL REPRESENTATIONS]
{\textsc{\footnotesize To appear in Mathematics of Computation}\\
\vskip1cm
{COMPUTATION OF CONFORMAL REPRESENTATIONS \\
OF COMPACT RIEMANN SURFACES}}

\author[G. L\'opez, D. Pestana, J. M. Rodr{\'\i}guez \and D. Yakubovich]
{Guillermo L\'opez Lagomasino$^{(1)}$, Domingo
Pestana$^{(1),(2)}$, Jos\'e M. Rodr{\'\i}guez$^{(1),(2)}$ \and
Dmitry Yakubovich$^{(3)}$}\thanks{\!\!\!\!\!\!\!\!$(1)\,\,\,$
Research partially supported by a grant from M.E.C. (MTM
2006-13000-C03-02) and a grant from U.C.III$\,$M./C.A.M.
(CCG07-UC3M/ESP-3339),
Spain.\\
$(2)\,\,\,$ Research partially supported by two grants from M.E.C.
(MTM 2006-11976 and MTM 2007-30904-E), Spain.\\
$(3)\,\,\,$ Research partially supported by the Grant
MTM2008-06621-C02-01, DGI-FEDER, of the Ministry of Science and
Innovation, Spain. }

\spb

\maketitle{}

\begin{abstract}
We find a system of two polynomial equations in two unknowns,
whose solution allows to give an explicit expression of the
conformal representation of a simply connected three sheeted
compact Riemann surface onto the extended complex plane. This
function appears in the description of the ratio asymptotic of
multiple orthogonal polynomials with respect to so called Nikishin
systems of two measures.
\end{abstract}

\


\spb

{\rm 2000 AMS Subject Classification:
30F99 (primary), 
05E35, 
30C30, 
58C15 
(secondary).}

\section{Introduction.}\label{sec:int}
Conformal representations of compact Riemann surfaces play an
essential role in approximation theory; in particular, they serve
to describe the so called ratio asymptotic of orthogonal
polynomials on the real line.

Let $\mu$ be a finite positive Borel measure whose compact support
$\supp(\mu)$ has infinitely many points and is contained in the
real line. By $q_n(w) = \kappa_nw^n+\cdots+\kappa_0$, with
$\kappa_n >0,$ we denote the $n$th orthonormal polynonomial with
respect to $\mu$; that is
\[ \int q_n(t) q_m(t) d\mu(t) = \delta_{n,m},
\]
where $\delta_{n,m}$ denotes the Kronecker delta.

It is well known and easy to verify that the sequence $\{q_n\}, n
\geq 0,$ satisfies a three term recurrence relation
\[  tq_n(t) = a_{n+1}q_{n+1}(t) + b_nq_n(t) + a_n q_{n-1}(t), \qquad
n \geq 1,
\]
$b_n \in {\mathbb{R}}, a_n >0$. The following result establishes a
close connection between ratio asymptotic of orthonormal
polynomials and the limit behavior of the recurrence coefficients
(see \cite{N}).
\begin{teorema} The following assertions are equivalent:
\begin{itemize}
\item $\displaystyle{\lim_n a_n = a > 0, \lim_n b_n = b}.$ \item
$\displaystyle{\lim_n \frac{q_{n+1}(w)}{q_n(w)} = \varphi(w) :=
\frac{w-b}{2a} + \sqrt{\left(\frac{w-b}{2a}\right)^2 -1}},$
uniformly on each compact subset of $\mathbb{C}\setminus
\supp(\mu)$.
\end{itemize}
If any of these conditions holds, then $\supp(\mu) = [b-2a,b+2a]
\cup e$, where $e$ is at most a denumerable set of isolated points
in $\overline{\mathbb{R}}\setminus [b-2a,b+2a].$
\end{teorema}

The function $z=\varphi(w)$ is the conformal representation of
$\overline{\mathbb{C}} \setminus [b-2a,b+2a]$ onto the complement
of the unit disk such that $\varphi(\infty) = \infty,
\varphi^{\prime}(\infty) > 0$ ($\overline{\CC}$ denotes the
extended complex plane). It is a solution of the algebraic
equation
\[ z^2 - \frac{w-b}{a} z +1 =0,
\]
which lives on a two sheeted compact Riemann surface. The numbers
$a$ and $b$ are determined by the first two coefficients of the
Laurent expansion of $\varphi$ at $\infty;$ in fact, $\varphi(w) =
w/a - b/a + {\mathcal{O}}(1/w), w \to \infty.$ In turn, the
algebraic equation which characterizes $\varphi$ is easy to
determine if the values $a,b$ are given. Any measure for which
$\supp(\mu) = [b-2a,b+2a] \cup e$, where $e$ is as described in
the theorem, verifies this ratio asymptotic if $\mu^{\prime}
> 0$ a.e. on $[b-2a,b+2a].$ This is known as the Rakhmanov-Denisov
theorem (see \cite{D} and \cite{R}).

Recently, (see \cite{ALR}, \cite{AKLR},  and \cite{LL}) results
analogous to those stated above and the Denisov-Rakhmanov theorem
were obtained for multiple orthogonal polynomials of Nikishin
systems of $m$ measures. In this case, the ratio asymptotic is
described in terms of a conformal representation of an $m+1$
sheeted compact Riemann surface onto the extended complex plane.
The reader interested in these results can check the references
given. We will not dwell into details and introduce directly the
relevant Riemann surface.

Let $\Delta_k, k=1,\ldots,m,$ be a system of bounded intervals of
the real line such that $\Delta_k \cap \Delta_{k+1} = \emptyset, k
=1,\ldots,m-1$. Consider the $(m+1)$-sheeted Riemann surface
$$
{\mathcal{R}}=\overline{\bigcup_{k=0}^m {\mathcal{ R}}_k} ,
$$
formed by the consecutively ``glued'' sheets
$$
{\mathcal {R}}_0:=\overline {\mathbb{C}} \setminus
{\Delta}_1,\quad {\mathcal R}_k:=\overline {\mathbb{C}} \setminus
( {\Delta}_k \cup
 {\Delta}_{k+1}),\,\, k=1,\dots,m-1,\quad {\mathcal R}_m=\overline
{\mathbb{C}} \setminus  {\Delta}_m,
$$
where the upper and lower banks of the slits of two neighboring
sheets are identified. In fact, this describes $\R$ as a branched
cover of $\overline{\CC}$. Denote by $\pi$ the corresponding
projection. Notice that $\R$ is a compact Riemann surface of genus
$0$; therefore, it is conformally equivalent to
$\overline{\mathbb{C}}$. Fix $l \in \{1,\ldots,m\}$. We are
interested in the conformal representation $\psi^{(l)}$,
 of $\mathcal{R}$ onto $\overline{\mathbb{C}}$ such that
\begin{equation} \label{eq:psi}
\psi^{(l)}(w) = w + {\mathcal{O}}(1)\,,\,\, w \to \infty^{(l)},
\quad \psi^{(l)}(w) = C_1/w + {\mathcal{O}}(1/w^2)\,,\,\, w \to
\infty^{(0)},
\end{equation}
where $w^{(l)}$ denotes the point in the sheet $l$ which projects
onto $w\in \overline{\mathbb{C}}$. When $m=1$ and $\Delta_1 =
[b-2a,b+2a],$ $\varphi=\frac 1a \psi^{(1)}|\R^{(1)}$.

In view of our motivation coming from the ratio asymptotic of
multiple orthogonal polynomials of Nikishin systems of measures,
we would like to find an explicit expression for $\psi^{(l)}$ or,
at least, an algebraic equation characterizing it, in terms of the
end points of the intervals $\Delta_k$. This problem seems to be
very difficult and unsolvable in all its generality, because it
requires the solution of high order algebraic equations. For this
reason, in the main part of the paper, we will restrict our
attention to the case $m=2$, where we have been able to achieve
our goals (see Theorem 3.1 below).

There is another way to look at the Riemann surface $\R$. If
$\psi:\R\to \overline{\CC}$ is any conformal representation and
$g:\overline{\CC}\to \R$ is its inverse, then
\begin{equation} \label{pi-circ-g}
G:=\pi\circ g
\end{equation}
is a holomorphic function from the extended Riemann sphere onto
itself and hence it is a rational function. Its degree (understood
as the number of solutions of the equation $G(z)=w$ for a generic
$w$) equals the number of sheets of the branched cover $\pi:\R\to
\overline{\CC}$, that is $m+1$. The end points of the intervals
$\Delta_k$ determine, in fact, \textit{the values of $G$ at its
critical points, where $G'(z)=0$. Therefore, we are interested in
determining a rational function $G$ of a given degree by its
critical values and, possibly, some additional topological
characteristics of the branched cover it defines.}

We remark that a related problem of classifying rational functions
with real critical values has been studied recently by various
authors, see  \cite{ShV} and references therein. It also turns out
that our problem is related to a problem of determining the
coefficients of a special Schwarz--Christoffel mapping of the
upper half-plane onto a special Riemann surface, see
\S\ref{computations} below.

The paper is divided as follows. In Section 2, we describe some
general properties of the conformal representations of
$\mathcal{R}$ onto $\overline{\mathbb{C}}.$ Sections 3 contains
the statement of the main results of the paper which are obtained
for the case when $m=2.$ Their proofs are carried out in Sections
4 and 5. The final Section 6 is devoted to the development of a
numerical algorithm for solving the system {{\bf (Syst)}} which in
turn allows one to calculate the functions $\psi^{(1)},
\psi^{(2)}$ numerically (for $m=2$).

\

{\bf Acknowledgements.} We would like to thank the referee for
carefully reading the manuscript and for some helpful suggestions.
We also express our gratitude to Daniel Est\'evez S\'anchez, who
helped us with the numerical calculations, using the MatLab
system. His work was subsidized by a Grant for the Support to
Study by the Autonomic Community of Madrid (2008).

\section{Some general properties}

Let $\psi, \Psi$ be any two conformal representations of
$\mathcal{R}$ onto $\overline{\mathbb{C}}$, then $\Psi \circ
\psi^{-1}$ is an automorphism of $\overline{\mathbb{C}}$.
Consequently, there exist constants $a,b,c,d,$ with $ad -cb \neq
0,$ such that
\[ \Psi(w) = \frac{a\psi (w)+b}{c\psi(w) + d}\,.
\]
Therefore, if we have found a conformal representation $\psi$, any
other $\Psi$ can be expressed explicitly in terms of $\psi$ and
the values of $\Psi$ at three distinct points (alternatively, in
terms of the values of $\Psi$ at two points and its first
derivative at one of the two points). For example (see
(\ref{eq:psi})),
\[ \psi^{(k)}(w) = C_2 \psi^{(l)}(w)/(\psi^{(l)}(w) -\psi^{(l)}(\infty^{(k)})),
\qquad k,l \in \{1,\ldots,m\}, \qquad k \neq l,
\]
where $C_2 \neq 0$ is an appropriate constant so that the Laurent
expansion of the right hand side at $\infty^{(k)}$ has leading
coefficient equal to $1$.

Let $\psi^{(l)}_{k}, k=0,1,\ldots,m,$ denote the branches of
$\psi^{(l)}$, corresponding to the different sheets of
$\mathcal{R}$. The function $\psi^{(l)}$ is the solution of the
algebraic equation
\begin{equation} \label{eq:algebraica} \prod_{k=0}^m (z - \psi^{(l)}_{k}(w)) = z^{m+1} +
\sum_{j=0}^{m} \alpha_j(w) z^{j} = 0.
\end{equation}
The coefficients $\alpha_j(w)$ are the so called symmetric
functions which are known to be entire functions on the complex
plane and can be expressed in terms of the branches of
$\psi^{(l)}$ through the Vieta relations.

By the definition of $\psi^{(l)}$, $\alpha_0(w) =
(-1)^{m+1}\prod_{k=0}^{m+1} \psi_k^{(l)}(w)$ has no singularity on
$\overline{\mathbb{C}}$; therefore, according to Liouville's
theorem it is constant. On the other hand, each $\alpha_j(w),
j=1,\ldots,m,$ has a simple pole at $\infty$; consequently, it is
a polynomial of first degree. In (\ref{eq:algebraica}), we can
solve for $w$ and from what was said before we find that $w = g(z)
= p(z)/q(z) = (\psi^{(l)})^{-1}(z),$ where $p$ is a polynomial of
degree $m+1$ and $q$ is a polynomial of degree $m$. The algebraic
equation which defines $\mathcal{R}$ is irreducible, therefore,
$(p,q) \equiv 1$; otherwise, $\psi^{(l)}$ would satisfy a lower
order algebraic equation in $z$. The poles of $g$ (which are
$\infty$ and the zeros of $q$) are the points
$\psi^{(l)}(\infty^{(k)}), k= 0,\ldots,m.$ The zeros of $g$ (which
are  the zeros of $p$) are the points $\psi^{(l)}(0^{(k)}), k=
0,\ldots,m$ (without loss of generality we can assume that $0
\not\in \cup_{k=0}^m \Delta_k).$ Since $\psi^{(l)}$ is single
valued, all these points are distinct. In particular, the zeros
and poles of $g$ are simple.

It is easy to verify that
\begin{equation} \label{eq:simetria} \psi^{(l)}(w) =
\overline{\psi^{(l)}(\overline{w})}, \qquad w \in \mathcal{R}.
\end{equation}
Indeed, let $\phi(w) := \overline{\psi^{(l)}(\overline{w})}$.
$\phi$ and $\psi^{(l)}$ have the same divisor; consequently, there
exists a constant $C$ such that $\phi= C\psi^{(l)}$. Comparing the
leading coefficients of the Laurent expansion of these functions
at $\infty^{(l)}$,  we conclude that $C=1.$

In terms of the branches of $\psi^{(l)}$, the symmetry formula
(\ref{eq:simetria}) indicates that  for each  $k= 0,1,\ldots,m,$
\begin{equation} \label{eq:sim1}
\psi^{(l)}_k: \overline{\mathbb{R}} \setminus ( {\Delta}_k\cup
{\Delta}_{k+1}) \longrightarrow \overline{\mathbb{R}}
\end{equation}
$( {\Delta}_0 =  {\Delta}_{m+1}=\emptyset)$; therefore, the
coefficients (in particular, the leading one) of the Laurent
expansions at $\infty$ of the branches of $\psi^{(l)}$ are real
numbers, and
\begin{equation}  \label{eq:sim2}
\psi^{(l)}_k(t_{\pm}) = \overline{\psi^{(l)}_k(t_{\mp})} =
\overline{\psi^{(l)}_{k+1}(t_{\pm})}, \qquad t \in  {\Delta}_{k+1}
\end{equation}
(the second equality is due to continuity).

\section{The results for $m=2$.}

Here we consider the problem of explicitly determining the
functions $G$ and $\psi^{(l)}$ for the case of two intervals
$\Delta_1, \Delta_2$. Using an affine transformation, if
necessary, we can assume that $\Delta_1 = [-\mu,-1]$ and $\Delta_2
= [1,\lambda]$, where $\lambda, \mu > 1$.

\begin{teorema}
\label{tnosimetrico} The functions $\psi^{(1)}$ and $\psi^{(2)}$
can be computed as follows:
$$
\psi^{(1)}=\frac {1}{H(a)}\,(1+G^{-1}) \,, \qquad \qquad
\psi^{(2)}=\frac{  A}{2H(a)}\, \frac{1+G^{-1}}{1-G^{-1}} \;,
$$
where $G(z):=H(z)/H(a)$,
\begin{equation}
\label{def H(z)} H(z) = h + z +\frac{Az}{1-z}+\frac {Bz}{1+z}\,,
\end{equation}
\begin{equation}
\label{def h} h=\frac14\,(a+\alpha)\Big(2a\alpha- \frac
{(a-\alpha)^2}{1-a\alpha}\Big) ,
\end{equation}
\begin{equation}
\label{def A,B} A=\frac14\,(1-\beta)(1-\alpha)(1-a)(1-b)\,, \qquad
B=\frac14\,(1+\beta)(1+\alpha)(1+a)(1+b)\,,
\end{equation}
$\beta$ and $b$ are the solutions of the equation
\begin{equation}
\label{eqn beta,b}
x^2+(a+\alpha)x+\frac{(a-\alpha)^2}{1-a\alpha}-3=0\,,
\end{equation}
verifying $\beta<-1$, $b>1$, and $\alpha$ and $a$ are the unique
solutions of the algebraic system
$$
\begin{cases}
2\,(a+\alpha)(3-a\alpha-a-\alpha)(3-a\alpha+a+\alpha)
+(\lambda-\mu) (a-\alpha)^3 =0\,
\\
(\lambda+\mu)^2(a-\alpha)^6=4\,(3+a\alpha)^3 (1-a\alpha)
(2+a+\alpha)(2-a-\alpha)\,,
\end{cases}
\leqno{\mathbf{(Syst)}}
$$
verifying $-1<\alpha <a<1$.
\end{teorema}

Since $H^{-1}(w)$ is the solution of the cubic equation
$$
z^3 - (w+A-B-h)z^2 -(1+A+B)z +w-h=0 \,,
$$
the functions $\psi^{(1)}$ and $\psi^{(2)}$ can be computed
explicitly if we know $\alpha$ and $a$.

\medskip

As follows from the proof, the rational function $G$, given by
this theorem, is given alternatively by \eqref{pi-circ-g}, where
$g$ is one of the conformal homeomorphisms of $\overline{\CC}$
onto $\R$. As part of the proof of Theorem \ref{tnosimetrico} we
will also obtain that $G$ and the real numbers $\beta,\alpha,a,b$
constitute a unique solution of the following system of relations
\begin{equation}    \label{relsG}
\begin{cases}
G=\frac PQ, \qquad \deg P=3, \; \deg Q=2, \\
\text{$-1$ and $1$ are poles of $G$}, \\
\lim_{z\to\infty} \frac {G(z)}z>0, \\
\beta<-1<\alpha<a<1<b, \\
G'(\beta)=G'(\alpha)=G'(a)=G'(b)=0,\\
-\mu= G(\b), \quad \lambda=G(b), \\
-G(\alpha)=G(a)=1.
\end{cases}
\end{equation}

\begin{teorema}
\label{tsistema} Consider the subsets of $\RR^2$
$$
\A=\{(a,\al): -1\le a<\al\le 1\}, \qquad \Omega= \{(\la,\mu):
\la\ge 1,\,\mu\ge 1\}.
$$
Then system {{\bf (Syst)}} defines a one-to-one correspondence
between these sets. Moreover, this mapping transforms the set
$\{(a,\al): -1 < a<\al< 1\}$ onto the interior of $\Om$,  the
sides $\{a=-1\}$, $\{\al=1\}$ of the triangle $\A$ onto the rays
$\{\la=1\},\, \{\mu=1\}\subset \Om$, respectively, and the vertex
$(a_0,\al_0)=(-1,1)$ of $\A$ onto the point $(1,1)\in \Om$.
\end{teorema}

If $\Delta_1$ and $\Delta_2$ have the same Euclidean length, then
$\mu=\lambda$, and we obtain a very simple result.

\begin{teorema}
\label{tsimetrico} If $\lambda = \mu$, then
$$
\psi^{(1)}= \frac{1+a^2}{2a^3}\,(G^{-1}+1) \,, \qquad \qquad
\psi^{(2)}=  \frac {(1-a^2)^2}{8a^3}\, \frac{G^{-1}+1}{G^{-1}-1}
\;.
$$
where $G(z):=H(z)/H(a)$,
$$
H(z) = z - \frac{(1-a^2)^2 z}{(1+a^2)(1-z^2)}\,,
$$
and $a$ is the unique solution on the interval $(0,1)$ of the
biquartic equation
$$
a^8 +(16 \lambda^2-8)a^6 +18a^4 -27=0\,.
$$
\end{teorema}

In this case, $H^{-1}(w)$ is the solution of the cubic equation
$$
z^3 - w z^2 + \frac{a^4-3a^2}{1+a^2}\, z + w = 0 \,.
$$

\

\section{
Proof of Theorem \ref{tnosimetrico}} \label{computations}

\subsection{
A geometric definition of $G$} Suppose intervals
$\Delta_1=[-\mu,-1]$ and $\Delta_2=[1,\l]$ \label{ss:geom G} are
given. Then
 there exists a
conformal homeomorphism  $g: \overline{\CC} \longrightarrow \R$
such that $G:=\pi\circ g: \overline{\CC} \longrightarrow
\overline{\CC}$ is a rational function $G=P/Q$ with real
coefficients, $\deg P=3$ and $\deg Q=2$. In fact, $g$ can be
constructed as follows. Consider a conformal map $g$ from the
closed upper halfplane $\overline\HH$ onto the simply connected
set
$$
\R^+:= \Big\{\zeta\in \R_0: \Im \pi(\zeta) \le 0\Big\} \bigcup
\Big\{\zeta\in \R_1: \Im \pi(\zeta) \ge 0\Big\} \bigcup
\Big\{\zeta\in \R_2: \Im \pi(\zeta) \le 0\Big\}.
$$
Notice that $\R^+$ is ``one half" of $\R$ and
$g\big(\overline\RR\big)$ coincides with the boundary of $\R^+$.
We can choose $g$ so that for some $\beta<z_1<\alpha<a<z_2<b$ we
have $g((-\infty,\b])=(-\infty,-\mu]^{(1)}$,
$g([\b,z_1])=(-\infty,-\mu]^{(0)}$,
$g([z_1,\a])=[-1,\infty)^{(0)}$, $g([\a,a])=[-1,1]^{(1)}$,
$g([a,z_2])=(-\infty,1]^{(2)}$, $g([z_2,b])=[\l,\infty)^{(2)}$ and
$g([b,\infty))=[\l,\infty)^{(1)}$.

By Schwarz's reflection principle, we can extend $g$ to a
conformal map (which we denote also by $g$) from $\CC$ into $\R$,
which is symmetric with respect to the real line. Then
$G:=\pi\circ g: \overline{\CC} \longrightarrow \overline{\CC}$ is
holomorphic; therefore, it is a rational function $G=P/Q$, where
$P$ and $Q$ have real coefficients, since $G$ maps $\overline\RR$
into $\overline\RR$. As $G$ has degree $3$ and  has three poles
$(z_1,z_2,\infty)$, then a fortiori $\deg P=3$ and $\deg Q=2$.

$G$ has two finite poles and four critical points on the real
line. Taking $G(d_1z+d_2)$ (with $d_1>0, d_2\in\RR$) instead of
$G(z)$, if necessary, without loss of generality we can assume
that the poles of $G$ are $-1$ and $1$. These normalizations
define $G$ uniquely.

The critical points of $G$ satisfy
$$
\beta<-1<\alpha<a<1<b.
$$
Moreover, $G$ increases on $(-\infty,\beta)$, $(\a,a)$,
$(b,\infty)$, and decreases on $(\beta,-1)$, $(-1,\a)$, $(a,1)$,
$(1,b)$. Therefore, $G(\a)< G(a)$; we also have
$-\mu=G(\b)<G(\a)=-1$ and $1=G(a)<G(b)=\lambda$.

It follows that $G$ satisfies all the relations \eqref{relsG}.

This way of constructing $G$ shows that it is in fact a kind of
Schwarz--Christoffel mapping from $\overline\HH$ onto the Riemann
surface $\R^+$, which may be viewed as an unbounded polygon, whose
angles are all equal to $2\pi$.

\subsection{
An auxiliary rational function $H$} \label{ss:H} In order to
compute $G$, let us consider the rational function
\begin{equation}
\label{H=cG} H(z):=c \,G(z)=c\, \frac {P(z)}{Q(z)}, \qquad \deg
P=3, \;\deg Q=2,
\end{equation}
where $c$ is a positive constant such that $\lim_{z\to\infty}
H(z)/z=1$.

Then, the requirements \eqref{relsG} on $G$ are equivalent to the
following conditions on the rational function $H$:
\begin{equation} \label{relsH}
\begin{cases}
H=\displaystyle \frac {P_1}Q, \quad \deg P_1=3, \; \deg Q=2, \\
\text{\rm $-1$ and $1$ are poles of $H$}, \\
\mu= H(\b)/H(\alpha), \quad \lambda= H(b)/H(a), \quad
H(\alpha)=-H(a), \\
\beta<-1<\alpha<a<1<b,  \\
H'(\beta)=H'(\alpha)=H'(a)=H'(b)=0, \\
\lim_{z\to\infty} H(z)/z=1.
\end{cases}
\end{equation}

Since $-1$ and $1$ are simple poles of $H$, we get from the two
last equations in \eqref{relsH} the equalities:
$$
H'(z)=\frac{(z-\beta)(z-\alpha)(z-a)(z-b)}{(z^2-1)^2} =
1+\frac{A}{(z-1)^2}+\frac {B}{(z+1)^2}\,.
$$
Therefore,
$$
H(z)=h+\int_0^z
\frac{(\zeta-\beta)(\zeta-\alpha)(\zeta-a)(\zeta-b)}{(\zeta^2-1)^2}\;
d\zeta = h + z +\frac{Az}{1-z}+\frac {Bz}{1+z}\,,
$$
where $h=H(0)$, so that \eqref{def H(z)} holds.

\subsection{
The deduction of system {{\bf (Syst)}}} \label{ss:deduct} We are
going to prove equations \eqref{def h}--\eqref{eqn beta,b} and the
validity of {\bf (Syst)}. In order to do that, we will express
$\beta,b,h,A$ and $B$ in terms of $\a$ and $a$ and write down
equations on $\a$, $a$. The choice of $\a$ and $a$ as unknown
variables seems to be the best one from the numerical point of
view, because they are bounded $(-1<\alpha<a<1)$.

From the equation for $H'(z)$ we deduce that
\begin{equation} \label{1}
(z-\beta)(z-\alpha)(z-a)(z-b)=(z^2-1)^2+A(z+1)^2+B(z-1)^2.
\end{equation}
Replacing $z=1$ and $z=-1$ in (\ref{1}), we obtain
\begin{equation}
\label{2.x} 4A=(1-\beta)(1-\alpha)(1-a)(1-b)\,, \qquad
4B=(1+\beta)(1+\alpha)(1+a)(1+b)\,.
\end{equation}
Identifying the coefficients in (\ref{1}), it follows that
\begin{align}
\label{a} & -\beta-\alpha-a-b=0 \,, \\
\label{b} & \beta\alpha+\beta a+\beta b+\alpha a+\alpha b+ab=-2+A+B \,, \\
\label{c} & -\beta\alpha a-\beta\alpha b-\beta a b-\alpha a b=2A-2B \,, \\
\label{d} & \beta\alpha a b=1+A+B \,.
\end{align}

Using (\ref{b}) and (\ref{d}), we get
\begin{equation} \label{betabdos}
\alpha (\beta+b) + a (\beta+b) + \beta b + \alpha a=-3+\beta\alpha
a b.
\end{equation}
Relation (\ref{a}) gives $\beta + b=-(\alpha+a)$, and we deduce
that
\begin{equation} \label{betab}
\beta b=\frac {a^2+\alpha^2+a \alpha-3}{1-a\alpha}= -3+\frac
{(a-\alpha)^2}{1-a\alpha} \,.
\end{equation}
Therefore, $\beta$ and $b$ can be obtained from $a$ and $\alpha$
as the solutions of the equation
\begin{equation}
\label{ecuacion} x^2+(a+\alpha)x+\frac
{(a-\alpha)^2}{1-a\alpha}-3=0\,.
\end{equation}

Since
$$
H(z)=h+z+\frac {Az}{1-z}+\frac {Bz}{1+z} \,,
$$
the condition $H(\alpha)=-H(a)$ means that
\begin{equation}\label{h}
4h+4\alpha+\alpha(1-\beta)(1-a)(1-b)+\alpha(1+\beta)(1+a)(1+b)
\end{equation}
\[ =
-4h-4a-a(1-\beta)(1-\alpha)(1-b)-a(1+\beta)(1+\alpha)(1+b).
\]
Using that $\beta + b=-(\alpha+a)$, we express $h$ in terms of $a$
and $\al$:
\begin{equation}
\label{2.y} 4h=(a+\alpha)(2a\alpha-3-\beta
b)=(a+\alpha)\Big(2a\alpha- \frac {(a-\alpha)^2}{1-a\alpha}\Big) .
\end{equation}

We are going to make use of the equations
\begin{equation}
\label{H(a)H(b)}
\begin{cases}
\la H(a)=H(b), \\
\mu H(a)=-H(\b).
\end{cases}
\end{equation}
Since $H(\b)<0<H(b)$, these equations are equivalent to the
following ones:
\begin{equation}
\label{H(a)+-H(b)}
\begin{cases}
\la H(\be)+\mu H(b)=0,\\
\;(\la+\mu)\,H(a)\;=H(b)-H(\b).
\end{cases}
\end{equation}

We have
\begin{align}
2H(a)&=2h+2a+\frac a2\,(1-\beta)(1-\alpha)(1-b)+\frac a2\,(1+\beta)(1+\alpha)(1+b) \notag \\
&=2h+a(3-\alpha(a+\alpha)+b\beta)\,, \notag \\
2H(b)&=2h+b\beta(a +\alpha)+b(3+a\alpha)\,, \notag \\
2H(\beta)&=2h+b\beta(a +\alpha)+\beta(3+a\alpha)\,. \notag
\end{align}
So we get from \eqref{H(a)+-H(b)} that
\begin{equation}
\label{system}
\begin{cases}
(\lambda+\mu)(2h+b\beta(a+\alpha))+(\lambda \beta+\mu b)(3+a\alpha)=0\,, \\
(\lambda+\mu)(2h+a(3-\alpha(a+\alpha)+b\beta))=(b-\beta)(3+a\alpha)\,.
\end{cases}
\end{equation}
Since $b$ and $\beta$ are the solutions of the equation
$x^2+(a+\a)x+b\b=0$, we have
$$
b=\frac {-(a+\alpha)+\sqrt{(a+\alpha)^2-4b\beta}}{2} \qquad
\hbox{and} \qquad \beta=\frac
{-(a+\alpha)-\sqrt{(a+\alpha)^2-4b\beta}}{2} \,.
$$
From this, we deduce that
$$
b-\beta= \sqrt{(a+\alpha)^2-4b\beta} \,,
$$
and
$$
\lambda \beta+\mu b= -\frac {\lambda+\mu}2\, (a+\alpha) + \frac
{\mu-\lambda}2\, \sqrt{(a+\alpha)^2-4b\beta}
$$
(one could, in fact, express here $b\beta$ in terms of $a$ and
$\a$, see \eqref{betab}). These last two equations and
(\ref{betab}) imply that (\ref{system}) is equivalent to:
\begin{equation}
\label{system2}
\begin{cases}
(\lambda+\mu)\Big(4h+\big(2\,\frac{(a-\alpha)^2}{1-a\alpha}-6\big)(a+\alpha)\Big)+\big(
-(\lambda+\mu) (a+\alpha)  +  (\mu-\lambda)
\sqrt{(a+\alpha)^2-4b\beta}\; \big)(3+a\alpha)=0\,,
\\
(\lambda+\mu)\Big(2h+a\Big(-\alpha(a+\alpha)+ \frac
{(a-\alpha)^2}{1-a\alpha}
\Big)\Big)=(3+a\alpha)\sqrt{(a+\alpha)^2-4b\beta}\;.
\end{cases}
\end{equation}
After replacing $4h$ by its value, given in \eqref{2.y}, we get
$$
 4h+\Big(2\,\frac{(a-\alpha)^2}{1-a\alpha}-6\Big)(a+\alpha) -
(a+\alpha)(3+a\alpha)
 = (a+\alpha)\Big(-9+a\alpha+\frac{(a-\alpha)^2}{1-a\alpha}\Big)
\,,
$$
and
\begin{equation}
\label{H(a)} 2H(a)=2h+a\Big(-\alpha(a+\alpha)+ \frac
{(a-\alpha)^2}{1-a\alpha} \Big) = \frac
{(a-\alpha)^3}{2(1-a\alpha)} \;.
\end{equation}
Then, (\ref{system2}) is equivalent to
\begin{equation}
\label{system3}
\begin{cases}
(\lambda+\mu)
(a+\alpha)\Big(-9+a\alpha+\frac{(a-\alpha)^2}{1-a\alpha}\Big) +
(\mu-\lambda) \sqrt{(a+\alpha)^2-4b\beta}\, (3+a\alpha)=0\,,
\\
(\lambda+\mu)\,\frac{(a-\alpha)^3}{1-a\alpha}=2\,(3+a\alpha)\sqrt{(a+\alpha)^2-4b\beta}\;.
\end{cases}
\end{equation}

Substituting the second equation in (\ref{system3}) into the first
one, we obtain an equivalent system:
\begin{equation}
\label{system4}
\begin{cases}
2\,(a+\alpha)\Big(-9+a\alpha+\frac{(a-\alpha)^2}{1-a\alpha}\Big)
+(\mu-\lambda) \frac {(a-\alpha)^3}{1-a\alpha} =0\,,
\\
(\lambda+\mu)\,\frac{(a-\alpha)^3}{1-a\alpha}=2\,(3+a\alpha)\sqrt{(a+\alpha)^2-4b\beta}\;.
\end{cases}
\end{equation}
The first equation in \eqref{system4} is equivalent to the first
equation in {\bf (Syst)}. Since $-1<\a<a<1$, both terms in the
second equation in (\ref{system4}) are positive. Therefore, this
equation is equivalent to
$$
(\lambda+\mu)^2\frac{(a-\alpha)^6}{(1-a\alpha)^2}=4\,(3+a\alpha)^2
\Big( (a+\alpha)^2+12-4\,\frac{(a-\alpha)^2}{1-a\alpha}\Big)\,,
$$
and hence to the second equation in {\bf (Syst)}.

We conclude that the functions $G$ and $H$ are necessarily given
by the formulas in Theorem \ref{tnosimetrico}, where $a$ and
$\alpha$ are \textit{some} solution of {{\bf (Syst)}}. The
uniqueness of the solution will be checked in the next subsection.
\subsection{The reverse arguments}
\label{ss:reverse} First we need some technical lemmas.

\begin{lema}
\label{l:pq}
 Let $t_1,t_2$ be the solutions of the real equation
  $t^2+pt+q=0$, with $\Re t_1\le \Re t_2$. Then $t_1<-1$ and $t_2>1$
  if and only if $1+|p|<-q$.
\end{lema}

\begin{proof}
It is obvious because the condition $t_1<-1$ and $t_2>1$ is
equivalent to $f(-1)<0$ and $f(1)<0$, where $f(t):=t^2+pt+q$.
\end{proof}

\begin{lema}
\label{l:betab}  Let $a$, $\a$ be any real numbers such that
$-1<\a<a<1$. Then the solutions $\beta, b$ of the equation
$$
x^2+(a+\alpha)x+\frac {(a-\alpha)^2}{1-a\alpha}-3=0
$$
are real. Assuming that
 $\beta\le b$,
one has $\beta<-1$, $b>1$.
\end{lema}

\begin{proof}
By Lemma \ref{l:pq}, it is sufficient to show that
$$
1+ |a+\a| < 3 - \frac {(a-\alpha)^2}{1-a\alpha}\;,
$$
whenever $-1<\a\le a<1$. By symmetry, it suffices to verify that
$$
W(a,\a):= a+\a + \frac {(a-\alpha)^2}{1-a\alpha} < 2\,, \qquad
\hbox{if $|\a|\le a<1$.}
$$
But this is a direct consequence of
$$
2-W(a,\a)= \frac {(1-a)(1-\alpha)(2+a+\alpha)}{1-a\alpha} >0\,.
$$
\end{proof}

\begin{lema}
\label{l:GH} Suppose $\la>1$, $\mu>1$ are arbitrary real numbers.
Let $\wt G$ be the unique rational function that satisfies
\eqref{relsG}, and define $\wt H$ from \eqref{H=cG}.

\textbf{(i)} There is at least one solution $(\a,a)$ to the system
{\bf (Syst)} such that $-1<\a<a<1$.

\textbf{(ii)} Given any such solution, define $\b<-1$, $b>1$ from
\eqref{eqn beta,b}, $A$, $B$, $h$ from \eqref{def A,B}, \eqref{def
h}, and then define the corresponding function $H$ from \eqref{def
H(z)}. Set $G(z)=H(z)/H(a)$. Then $G=\wt G$, $H=\wt H$.
\end{lema}

\begin{proof}
Part (i) has been already checked in subsection \ref{ss:deduct}.

Now assume that $\a,a,b,\b,H,G$ has been constructed as in (ii).
We show that
\begin{equation} \label{H0prima}
H'(z) = 1+\frac{A}{(z-1)^2}+\frac {B}{(z+1)^2}
=\frac{(z-\beta)(z-\alpha)(z-a)(z-b)}{(z^2-1)^2} \,.
\end{equation}
It suffices to check  equations (\ref{a})--(\ref{d}).

We have (\ref{a}) from the definition of $\b$ and $b$. This
definition also gives (\ref{betab}), which implies
(\ref{betabdos}). This last equation is the difference of
(\ref{b}) and (\ref{d}). The definition of $A$ and $B$ gives
$$
2A+2B=1+\b\a+\b a+\b b + \a a + \a b + a b +\b \a a b \,,
$$
and this equation is the sum of (\ref{b}) and (\ref{d}).
Therefore, we can also obtain the equations (\ref{b}) and
(\ref{d}). The definition of $A$ and $B$ also gives (\ref{c}):
$$
\begin{aligned}
-2A+2B&=\b+\a+a+b+\b \a a+\b \a b + \b a b + \a   a b\\
&=\b \a a+\b \a b + \b a b + \a   a b \,.
\end{aligned}
$$
Consequently, we have proved (\ref{H0prima}).

The definition of $h$ and formulas
$$
\b + b = -a -\a \,, \qquad \b b = \frac
{(a-\alpha)^2}{1-a\alpha}-3\,,
$$
give (\ref{h}), which is equivalent to $H(\a)=-H(a)$.

We can reverse the arguments given in Section \ref{ss:geom G} to
show implications {{\bf (Syst)}} $\Rightarrow$ \eqref{system3}
$\Rightarrow$ \eqref{H(a)+-H(b)}. We also see that \eqref{H(a)}
holds, which, together with the second equation in system
\eqref{H(a)+-H(b)}, implies that $H(a)\ne 0$, $H(b)\ne H(\b)$.
Therefore, we can deduce \eqref{H(a)H(b)} from \eqref{H(a)+-H(b)}.
Hence all the conditions \eqref{relsH} hold.

Now, we can assert that $H$ increases on $(-\infty,\beta)$,
$(\a,a)$, $(b,\infty)$, and   decreases on $(\beta,-1)$,
$(-1,\a)$, $(a,1)$, $(1,b)$. Since $H$ increases on $(\a,a)$, we
have $H(\a)<H(a)$, but $H(\a)=-H(a)$, so we deduce that $H(a)>0$.

Then $H'(\b)=H'(\a)=H'(a)=H'(b)=0$,
$H\big(\overline{\RR}\big)=\overline{\RR}$ and $\lim_{z\to \infty}
H(z)/z=1$. Furthermore, $H$ has just three poles; namely $-1$, $1$
and $\infty$, and they are simple.

The function $G(z)=H(z)/H(a)$ verifies $G(\b)=-\mu$, $G(\a)=-1$,
$G(a)=1$ and $G(b)=\lambda$. We also have
$G'(\b)=G'(\a)=G'(a)=G'(b)=0$ and
$G\big(\overline{\RR}\big)=\overline{\RR}$.

Since $H(a)>0$, the function $G$ increases on $(-\infty,\beta)$,
$(\a,a)$, $(b,\infty)$, and  decreases on $(\beta,-1)$, $(-1,\a)$,
$(a,1)$, $(1,b)$.

$G$ is a rational function of degree $3$; therefore, we conclude
that $G$ gives rise to a conformal map $g$ of $\overline{\CC}$
onto a Riemann surface $\R^*$ with three sheets $(\R^*,
\R_1^*,\R_2^*)$, so that $G=\pi\circ g$, where $\pi:\R^*\to
\overline{\CC}$ is the canonical projection. The set of branch
points of $\R^*$ is $\{G(\b), G(\a), G(a), G(b)\}=\{-\mu, -1, 1,
\l\}$. All these branch points have order two, because $G'$ has
simple zeros on $\b,\a,a,b$.

Since $\R^*=\R^* \cup \R_1^* \cup \R_2^*$ is connected and the
branch points have order two, there is some $\R_j^*$ (for
instance, $\R_1^*$) with two cuts and $\R^*, \R_2^*$ have just one
cut. By the monotonicity properties of $G$, we can deduce that
$G(\infty)=\infty^{(1)}$, $G(-1)=\infty^{(0)}$,
$G(1)=\infty^{(2)}$, and that $G\big(\overline{\HH}\big) =
(\R^*)^+$, where $(\R^*)^+$ is the union of one half of each
$\R_j^*$ ($j=0,1,2$). We also have $G\big(\overline{\RR}\big) =
\partial (\R^*)^+$. Every point of $\overline{\RR}\setminus
([-\mu,-1]\cup [1,\lambda])$ has $3$ preimages by $G$ in
$\overline{\RR}$, and every point of $(-\mu,-1)\cup (1,\lambda)$
has one preimage by $G$ in $\overline{\RR}$. Since $(\R^*)^+$ is
the union of one half of each $\R_j^*$ and it is connected, the
cuts are $(-\mu,-1)\cup (1,\lambda)$. Consequently, $\R^*=\R$.

Therefore, $G^{-1} \circ \wt G$ is a conformal map from
$\overline\CC$ onto $\overline\CC$, and it must be a M\"obius
transformation. Since $G(\infty)=\wt G(\infty)=\infty^{(1)}$,
$G(-1)=\wt G(-1)=\infty^{(0)}$, $G(1)=\wt G(1)=\infty^{(2)}$, the
M\"obius map $G^{-1} \circ \wt G$ fixes the points $\infty$, $-1$
and $1$. It follows that $G^{-1} \circ \wt G$ is the identity map,
and therefore $G = \wt G$.
\end{proof}

\begin{corolario}
The solution $(\a,a)$ to system {{\bf (Syst)}} with the properties
$-1<\a<a<1$ is unique.
\end{corolario}

\begin{proof} Indeed, as shown above, any solution $(\a,a)$ of this system gives rise to
a function $G$, which coincides with $\wt G$. Denote the
parameters that correspond to $\wt G$ by $\tilde \a$, $\tilde a$,
etc. Then $\a=\tilde\a$ and $a=\tilde a$ (because they are the
only two critical points of $G$ and $\wt G$ on the interval
$(-1,1)$).
\end{proof}

\

Since $G(z)=H(z)/H(a)$, we obtain that $G^{-1}(w)=H^{-1}(H(a)w)$.
We have
$$
w =H(z) =\frac{ (h + z)(1-z^2) + A(1+z)z + B(1-z)z}{1-z^2}\,,
$$
or equivalently,
\begin{equation}
\label{2.z} z^3 - (w +A-B-h)z^2 -(1+A+B)z +w -h =0 \,.
\end{equation}
Observe that (\ref{2.z}) allows one to obtain an explicit
expression for $H^{-1}$ (and then for $\psi^{(1)}$ and
$\psi^{(2)}$), once we have solved {\bf (Syst)}.

\medskip

We compute now $\psi^{(1)}$ and $\psi^{(2)}$ in terms of $G$.
Recall that $\psi^{(1)}$ and $\psi^{(2)}$ are determined by
$$
\begin{aligned}
& \psi^{(1)}(w)= w+{\mathcal{O}}(1),\,\, w\rightarrow\infty
^{(1)}, \qquad \psi^{(1)}\big(\infty^{(0)}\big)=0\,,
\\
& \psi^{(2)}(w)= w+{\mathcal{O}}(1),\,\, w\rightarrow\infty
^{(2)}, \qquad \psi^{(2)}\big(\infty^{(0)}\big)=0\,.
\end{aligned}
$$

Note that $G^{-1}(\infty^{(0)})=-1$, $G^{-1}(\infty^{(1)})=\infty$
and $G^{-1}(\infty^{(2)})=1$. Hence, $\psi^{(1)}=c_1(1+G^{-1})$
for some constant $c_1$. In order to calculate the constant $c_1$,
notice that
$$
\begin{aligned}
& H(z)=z+{\mathcal{O}}(1),\qquad z\rightarrow\infty, \\
& H^{-1}(w)=w+{\mathcal{O}}(1),\qquad w\rightarrow\infty^{(1)}, \\
& G^{-1}(w)=H(a)\,w+{\mathcal{O}}(1),\qquad w\rightarrow\infty^{(1)}, \\
& w+{\mathcal{O}}(1)=c_1 H(a)\,w+{\mathcal{O}}(1),\qquad
w\rightarrow\infty^{(1)},
\end{aligned}
$$
and then
$$
c_1= \frac {1}{H(a)}\;.
$$

We also have $\psi^{(2)}=c_2(1+G^{-1})/(1-G^{-1})$ for some
constant $c_2$. In order to compute the constant $c_2$, let us
expand the left hand side of (\ref{2.z}) in powers of $(z-1)$. One
gets that (\ref{2.z}) is equivalent to
$$
 \begin{aligned}
(z-1)^3 +(3-w-A+B+h)(z-1)^2 +(2-2w-3A+B+2h)(z-1)-2A&=0\,, \\
\frac {(z-1)^2}w +\frac {(3-w-A+B+h)(z-1)}w +\frac
{2-2w-3A+B+2h}w-\frac {2A}{w(z-1)}&=0\,.
 \end{aligned}
$$

Since $H^{-1}\big(\infty^{(2)}\big)=1$, we have that
$$
\lim_{w\to\infty^{(2)}} \frac {2A}{w(1-H^{-1}(w))}=
\lim_{w\to\infty^{(2)}} \frac {2A}{w(1-z)}=2 \,,
$$
or, equivalently,
$$
\lim_{w\to\infty^{(2)}} \frac {2}{w(1-G^{-1}(w))}=
\lim_{w\to\infty^{(2)}} \frac {2}{w(1-H^{-1}(H(a)w))}=
\lim_{w\to\infty^{(2)}} \frac {2H(a)}{w(1-H^{-1}(w))}=  \frac
{2H(a)}{A}\;.
$$
Recall that
$$
1 = \lim_{w\to\infty^{(2)}} \frac {\psi^{(2)}(w)}{w}= c_2
\lim_{w\to\infty^{(2)}} \frac {1+G^{-1}(w)}{w(1-G^{-1}(w))}= c_2
\lim_{w\to\infty^{(2)}} \frac {2}{w(1-G^{-1}(w))}=  c_2\; \frac
{2H(a)}{A}\;,
$$
and this gives $ c_2= A/2H(a). $ This finishes the proof of
Theorem \ref{tnosimetrico}.

\section{
Proofs of Theorems \ref{tsistema} and \ref{tsimetrico}}
\label{2thms}

\begin{proof}[Proof of Theorem \ref{tsistema}]
We associate with any point $(\la,\mu)\in \RR^2$ the point
$(u,v)$, where \beqn \label{uv} u=\la-\mu, \qquad v=(\la+\mu)^2.
\neqn In terms of $u,v$, {{\bf (Syst)}} may be rewritten as \beqn
\begin{aligned}
\label{alg-syst2} \displaystyle
\begin{cases}
u=-\displaystyle \frac {
2\,(a+\alpha)(3-a\alpha-a-\alpha)(3-a\alpha+a+\alpha)}
{ (a-\alpha)^3 }\,, \\
\phantom{=} \\
v=4 \,\displaystyle \frac{(3+a\alpha)^3
(1-a\alpha)(2+a+\alpha)(2-a-\alpha)} {(a-\alpha)^6}\, .
\end{cases}
\end{aligned}
\neqn One can check that these equations imply that\footnote{We
found these formulas by using the package Maple of symbolic
calculations.}
\begin{align}
\label{exprs:u,v}
v-(2-u)^2= 16\, \frac { (a^2-1)(\al^2+2 a\al-3)^3 }{(a-\al)^6}, \qquad 
v-(2+u)^2= 16 \,\frac { (\al^2-1)(a^2+2 a\al-3)^3 }{(a-\al)^6}.
\end{align}
Notice that $\al^2+2 a\al-3< 0 $ for all $(a,\al)\in (-1,1)\times
(-1,1)$. It also follows that $\al^2+2 a\al-3\le0 $ for all
$(a,\al)\in [-1,1]\times [-1,1]$.

Take any point $(a,\al)\in \A$. By \eqref{exprs:u,v}, the
corresponding pair $(u,v)$, given by \eqref{alg-syst2}, satisfies
$v\ge (2-u)^2$, $v\ge (2+u)^2$. Hence $v\ge0$, and $\sqrt{v}\ge
2-u$, $\sqrt{v}\ge 2+u$. Put \beqn \label{uv->lamu} \la=\frac 12\,
(u+\sqrt{v}), \qquad \mu=\frac 12 \,(\sqrt{v}-u); \neqn then
$\la,\mu\ge 1$. By examining \eqref{uv}, one gets that the pair
$(\la, \mu)$ is the unique solution of {{\bf (Syst)}} that belongs
to $\Omega$. By repeating the same arguments with strict
inequalities, one also sees that $\la> 1$, $\mu> 1$ whenever
$(a,\al)\in \inter\A$. The assertions of the Theorem about the
images in $\Omega$ of the sets $\{\al=-1\}$, $\{a=1\}$, and
$\{\al=-1,\; a=1\}$ follow easily from \eqref{exprs:u,v}.

It was proved already in the previous section that for any
$(\la,\mu)\in\inter \Om$, {{\bf (Syst)}} has a unique solution
$(a,\al)\in \inter \A$. This finishes the proof of the Theorem.
\end{proof}

\begin{proof}[Proof of Theorem \ref{tsimetrico}]
Assume that the intervals $[-\mu,-1]$ and $[1,\l]$ are symmetric,
i.e. $\l=\mu$. Then, taking into account that
$(3-a\alpha-a-\alpha)(3-a\alpha+a+\alpha)>0$, we deduce from the
first equation in {\bf (Syst)} that $\alpha=-a$. Hence (\ref{2.y})
gives $h=0$. By (\ref{ecuacion}), it follows that $\b$ and $b$ are
the solutions of the equation
$$
x^2=3-\frac{4a^2}{1+a^2}=\frac{3-a^2}{1+a^2}\;.
$$
Consequently, $\beta=-b$. Hence, by (\ref{2.x}), $\displaystyle
B=A= - \frac{(1-a^2)^2}{2(1+a^2)}\,. $ Therefore,
$$
H(z) = z - \frac{(1-a^2)^2 z}{(1+a^2)(1-z^2)}\,, \qquad H(a)=\frac
{2a^3}{1+a^2}\,,
$$
and
$$
\psi^{(1)}=\frac{1+a^2}{2a^3}\,(G^{-1}+1) \,, \qquad \qquad
\psi^{(2)}= \frac {(1-a^2)^2}{8a^3}\, \frac{G^{-1}+1}{G^{-1}-1}
\;.
$$
Consequently, $H$ is an odd function, and the second equation in
{\bf (Syst)} yields
$$
16 \lambda^2 a^6 =(1+a^2)(3-a^2)^3 =27-18a^4+8a^6-a^8 .
$$

We have obtained that $a$ is solution of
\begin{equation}
\label{simetria} a^8 +(16 \lambda^2-8)a^6 +18a^4 -27=0\,.
\end{equation}
In this simpler case, we can check directly that there is exactly
one solution of (\ref{simetria}) on the interval $(0,1)$. In fact,
the function $u(t)=t^4 +(16 \lambda^2-8)t^3 +18t^2 -27$ verifies
$u(0)=-27<0, u(1)=16 \lambda^2-16>0$, and $u'(t)=4t^3 +3(16
\lambda^2-8)t^2 +36t > 36t >0$ for every $t \in (0,1)$.
Furthermore, since (\ref{simetria}) is a biquartic equation, it
can be solved by radicals.
\end{proof}

\section{Numerical resolution of the nonlinear system.}
\label{sec:num-res}

The purpose of this section is to show that the problem of solving
numerically the algebraic system {{\bf (Syst)}} is a simple task,
both from the theoretical and practical points of view.

Let us use the vector notation: $A=(\alpha,a)\in \inter\A$ and
$L=(\lambda,\mu)\in \inter\Omega$. By Theorem \ref{tsistema},
{{\bf (Syst)}} defines uniquely a map
$$
F:A\mapsto L.
$$
We suggest to use a continuation method for solving this equation
numerically, which is a variation of methods described, for
instance, in \cite{DeufHo}, Section 4.4.2. Notice that these
methods go back to the treatise \cite{Poi} by H. Poincar\'e on
celestial mechanics.

Recall that the Newton method for solving equation $F(A)=\wh L$
with the starting point $A_{st}$ consists in the following (see,
for instance, \cite{DeufHo}, Section 4.2). Fix a small parameter
$\sigma>0$ and define iteratively  $A^{(0)}=A_{st}$;
$A^{(n+1)}=A^{(n)}-F'(A^{(n)})^{-1}F(A^{(n)})$, $n\ge0$. This
iterative process finishes when $\big| F(A^{(n)})-\wh L
\big|<\sigma$, and the last point $A_{fin}:=A^{(n)}$ is taken for
an approximate solution of equation $F(A)=\wh L$.

In the continuation scheme, we apply the Newton algorithm several
times. Given a point $L_*\in\inter\Om$, we wish to find a good
approximation for a solution $A_*\in\inter\A$ of the equation
$F(A_*)=L_*$. The method goes as follows.

\textbf{Part 1.} Choose an initial approximation $A_0$ of the
solution. We do it solving the equation $F(A_0)=L_0$, where
$L_0=(\lambda_*+\mu_*,\lambda_*+\mu_*)/2$ is a symmetric vector,
by applying Theorem \ref{tsimetrico}.

\textbf{Part 2.} Choose a large integer $n>0$, and divide the
interval $[L_0,L_*]\subset \Om$ into $n$ equal subintervals by
division points $L_0, L_1, \dots, L_n=L_*$. So, we set
$L_k=\big((n-k)L_0+kL_*\big)/n$.

\textbf{Part 3.} Let the parameter $\sigma>0$ be fixed. The
calculation is performed in $n$ steps, applying the Newton method
$n$ times.

On the $k$th step ($1\le k\le n$), we find an approx\-im\-ate
solution $A^{(k)}_{fin}$ of the equation $F(A)=L_k$ by running the
above-described Newton method with the starting point
$A^{(k)}_{st}=A^{(k-1)}_{fin}$, which is available from the
$(k-1)$th step. If $k=1$, then we put $A^{(1)}_{st}=A_0$, where
$A_0$ is the value found from Part 1.

Once all $n$ steps of Part 3 are performed successfully, one takes
$A^{(n)}_{fin}$ for an approximation of the solution $A_*$ of the
equation $F(A_*)=L_*$. One has, in fact:
$F(A^{(n)}_{fin})\thickapprox L_n=L_*$.

The applicability of this method for our concrete function $F$ is
justified by the following theorem. We will formulate it in a more
general setting.

\begin{teorema}
\label{t:cont-meth} Let $\wt\A$ and $\wt\Om$ be open subsets of
$\RR^d$, and suppose that $\wt\A$ is connected and $\wt\Om$ is
convex. Let $F$ be a $\mathcal{C}^2$ homeomorphism from $\wt\A$ to
$\wt\Om$ such that $\det F'\ne 0$ in $\wt\A$. Then the above
continuation scheme for the Newton method is numerically feasible
in the following sense.

Let a point $L_*\in \inter\Om$ and a starting point $L_0$ be
given. Let $A_*\in\wt \A$ be the (unique) solution of the equation
$F(A)=L_*$. Then for any $\eps>0$ there exist a $\sigma_0>0$ and
an integer $N_0\ge 1$ such that for any $n>N_0$ and any
$\sigma<\sigma_0$, the Newton method stops on each of $n$ steps of
Part 3 of the algorithm, and the approximate solution
$A^{(n)}_{fin}=A^{(n)}_{fin}(\sigma)$ obtained satisfies
$|A^{(n)}_{fin}-A_*|<\eps$.
\end{teorema}

Because of the topological assumption on $F$, the situation is
very simple. The feasibility of the method is closely related with
the continuation property for $F$, see \cite{OrtRh}, Section 5.3.

It seems that this theorem is a variation of classical results,
known to specialists. Since the authors were unable to find an
exact reference, a sketch of the proof is included.

\begin{proof}[Proof of Theorem \ref{t:cont-meth}]
Denote by $B(q,r)$ the open ball centered at a point $q\in\RR^d$
of radius $r$. We apply Theorem 4.10 from \cite{DeufHo}. It gives
sufficient quantitative conditions for convergence of the Newton
method, which imply, in particular, the following. Suppose $\wh
A\in\wt\A$, and let $\wh L=F(\wh A)\in\wt\Om$. Then, by the
assumption, $\det F'(\wh A)\ne 0$. It follows that there exists
$\de>0$ with the following property.

{$\bf (*)$} Given any $\eps>0$, there exists $\sigma_0>0$ such
that to any $\sigma<\sigma_0$ there corresponds an integer
$M(\sigma)$ satisfying the following: the Newton method for
solving equation $F(A)=\wh L$ with any starting point $A_{st}\in
B(\wh A,\de)$ stops after at most $M(\sigma)$ steps, and the
approximate solution $A_{fin}(\sigma)$ obtained fulfils \linebreak
$|A_{fin}(\sigma)-\wh A|<\eps$.

\

Next, the pre-image $F^{-1}([L_0,L_*])$ is compact, hence the
infimum of $|\det F'|$ on this set is positive. It follows that
the above condition (*) takes place uniformly for all pairs of
points $(\wh A, \wh L)\in \wt\A\times\wt\Om$ such that $F(\wh
A)=\wh L$ and $\wh L\in [L_0,L_*]$.

Fix $L_0$, $L_*$, and fix some $\de$ such that (*) holds for all
pairs $(\wh A, \wh L)$ as above. Take any $\eps>0$, and let us
prove that the conclusions of the theorem hold for this $\eps$.
We can assume that $\eps<\de/2$. Let us find the corresponding
$\sigma_0$ such that $(*)$ holds. There exists a positive $\rho$
such that $|A'-A''|<\eps$ whenever $F(A'),F(A'')\in [L_0,L_*]$ and
$|F(A')-F(A'')|<\rho$. Let $N_0$ be any integer such that
$hN_0>|L_*-L_0|$. We claim that the assertion of the Theorem holds
for these $\sigma_0$ and $N_0$.

Indeed, take any $n\ge N_0$, and put (as above)
$L_k=\big((n-k)L_0+kL_*\big)/n$, and $A_k=F^{-a}(L_k)$. Then
$|A_k-A_{k+1}|<\eps$ for all $k$. Consider the properties:
$$
{\bf(\Ga)}_k \quad |A_{st}^{(k)}-A_k|<\de; \qquad \qquad
{\bf(\De)}_k \quad |A_{fin}^{(k)}-A_k|<\eps.
$$
\noindent Since $A_{st}^{(k)}=A_0$, ${\bf(\Ga)}_1$ holds. Next,
for any $k$, $(\Ga)_k$ implies $(\De)_k$, due to $(*)$, and
$(\De)_k$ implies $(\Ga)_{k+1}$, because
$$
|A_{st}^{(k+1)}-A_{k+1}|= |A_{fin}^{(k)}-A_{k+1}| \le
|A_{fin}^{(k)}-A_k| + |A_k-A_{k+1}|<2\eps<\de.
$$
By induction, we obtain that $(\Ga)_k$ and $(\De)_k$ hold for all
$k$. In particular,
$$
|A_{fin}^{(n)}-A_n| = |A_{fin}^{(n)}(\sigma)-A_*|<\eps.
$$
\end{proof}

Let us return to our particular function $F$, defined by system
{{\bf (Syst)}}. It follows from Theorem \ref{tsistema} and
explicit formulas \eqref{alg-syst2}, \eqref{uv->lamu} that
$F:\inter \A\to \inter \Om$ is a $\mathcal{C}^\infty$ smooth
homeomorphism. Put $\wt\A=\inter\A$ and $\wt\Om=\inter \Om$. In
order to prove that the above Theorem \ref{t:cont-meth} applies to
$F$, it only remains to check that $\det F'$ does not vanish in
$\inter \A$. This follows from an explicit calculation, using
formulas \eqref{alg-syst2}, \eqref{uv->lamu}, that represent $F$
as a composition map. One has
$$
\det F'=\det \frac {\p(\la,\mu)}{\p(a,\alpha)} =\det \frac
{\p(\la,\mu)}{\p(u,v)}\cdot \det \frac {\p(u,v)}{\p(a,\alpha)}.
$$
Then, $\det \frac {\p(\la,\mu)}{\p(u,v)}=1/(4v^{1/2})>0$, because
$v=v(a,\alpha)>0$ in $\inter \A$. At last, \eqref{alg-syst2}
implies  that
$$
\det \frac {\p(u,v)}{\p(a,\alpha)} = -2^7\,{(a-\al)^{-10}} {(a^2+2
a\al-3)^2 (3+a\al)^2 (\al^2+2 a\al-3)^2}
$$
(we have applied the Maple package here). As we noted in the proof
of Theorem \ref{tsistema}, $a^2+2 a\al-3<0$ and $\al^2+2 a\al-3<0$
for all $(a,\al)\in\inter\A$. Therefore $\det \frac
{\p(u,v)}{\p(a,\alpha)}<0$ and $\det F'(A)<0$ for all
$A=(a,\al)\in\inter\A$.

In the next table, we reproduce some numerical results of the
implementation of the above method.

\newpage

\begin{table}[h]
\caption{Numerical results}
\begin{center}
\begin{tabular}{|c|c|c|c|c|c|}
\hline
$\la$  &  $\mu$  &  $\beta$ &   $\alpha$ &    $a$  &     $b$ \\
\hline \hline
  1.01 &   1.10 & -1.02433457 & -0.97566543 & 0.99756619 & 1.00243381\\
\hline
  1.01 &   1.50 & -1.11120778 & -0.88879220 & 0.99776894 & 1.00223104\\
\hline
  1.01 &   2.00 & -1.20162987 & -0.79837009 & 0.99796298 & 1.00203698\\
\hline
  1.01 &   5.00 & -1.53512664 & -0.46487320 & 0.99855894 & 1.00144090\\
\hline
  1.01 &  10.00 & -1.80340393 & -0.19659585 & 0.99893546 & 1.00106432\\
\hline
  1.01 &  20.00 & -2.05587223 & 0.05587246 & 0.99922942 & 1.00077035\\
\hline
  1.01 &  50.00 & -2.33815403 & 0.33815423 & 0.99950547 & 1.00049433\\
\hline
  1.01 & 100.00 & -2.50602945 & 0.50602961 & 0.99964857 & 1.00035127\\
\hline
  1.10 &   1.50 & -1.10904414 & -0.89095444 & 0.97812791 & 1.02187067\\
\hline
  1.10 &   2.00 & -1.19803860 & -0.80195744 & 0.97999986 & 1.01999618\\
\hline
  1.10 &   5.00 & -1.52833535 & -0.47164952 & 0.98579364 & 1.01419123\\
\hline
  1.10 &  10.00 & -1.79584373 & -0.20413551 & 0.98948341 & 1.01049582\\
\hline
  1.10 &  20.00 & -2.04866217 & 0.04868414 & 0.99237698 & 1.00760105\\
\hline
  1.10 &  50.00 & -2.33227827 & 0.33229733 & 0.99510249 & 1.00487845\\
\hline
  1.10 & 100.00 & -2.50132594 & 0.50134152 & 0.99651797 & 1.00346645\\
\hline
  1.50 &   2.00 & -1.18411168 & -0.81581475 & 0.90730105 & 1.09262537\\
\hline
  1.50 &   5.00 & -1.50111605 & -0.49858460 & 0.93306652 & 1.06663413\\
\hline
  1.50 &  10.00 & -1.76499059 & -0.23458419 & 0.95001591 & 1.04955887\\
\hline
  1.50 &  20.00 & -2.01887210 & 0.01933234 & 0.96355662 & 1.03598315\\
\hline
  1.50 &  50.00 & -2.30775501 & 0.30816158 & 0.97647774 & 1.02311569\\
\hline
  1.50 & 100.00 & -2.48160429 & 0.48193938 & 0.98324158 & 1.01642333\\
\hline
  2.00 &   5.00 & -1.47250352 & -0.52657019 & 0.87477665 & 1.12429705\\
\hline
  2.00 &  10.00 & -1.73155665 & -0.26707834 & 0.90556526 & 1.09306973\\
\hline
  2.00 &  20.00 & -1.98589261 & -0.01259181 & 0.93068264 & 1.06780179\\
\hline
  2.00 &  50.00 & -2.28012155 & 0.28148839 & 0.95501596 & 1.04361720\\
\hline
  2.00 & 100.00 & -2.45919974 & 0.46033681 & 0.96787312 & 1.03098981\\
\hline
  5.00 &  10.00 & -1.59758021 & -0.39392722 & 0.70769711 & 1.28381032\\
\hline
  5.00 &  20.00 & -1.84539761 & -0.14406679 & 0.77852683 & 1.21093757\\
\hline
  5.00 &  50.00 & -2.15584295 & 0.16630302 & 0.85234419 & 1.13719574\\
\hline
  5.00 & 100.00 & -2.35580726 & 0.36490617 & 0.89323258 & 1.09766851\\
\hline
 10.00 &  20.00 & -1.70949675 & -0.26791189 & 0.61012964 & 1.36727900\\
\hline
 10.00 &  50.00 & -2.02314506 & 0.04811583 & 0.73131282 & 1.24371641\\
\hline
 10.00 & 100.00 & -2.23983688 & 0.26275325 & 0.80256032 & 1.17452331\\
\hline
 20.00 &  50.00 & -1.86312347 & -0.09255784 & 0.56794501 & 1.38773629\\
\hline
 20.00 & 100.00 & -2.09037160 & 0.13418100 & 0.67474116 & 1.28144945\\
\hline
 50.00 & 100.00 & -1.85498816 & -0.06912394 & 0.44519539 & 1.47891671\\
\hline
\end{tabular}
\end{center}
\end{table}

\

\begin{obeylines}
{\sc Guillermo L\'opez, Domingo Pestana, and Jose Manuel Rodr{\'\i}guez
Dept. of Mathematics, Universidad Carlos III de Madrid}
E-mail addresses: {\tt lago@math.uc3m.es, dompes@math.uc3m.es, jomaro@math.uc3m.es}
\end{obeylines}

\

\begin{obeylines}
{\sc Dmitry Yakubovich,
Dept. of Mathematics, Universidad Aut\'{o}noma de Madrid}  and
{\sc Instituto de Ciencias Matem\'{a}ticas (CSIC-UAM-UC3M-UCM)}
E-mail address: {\tt dmitry.yakubovich@uam.es}
\end{obeylines}


\begin{thebibliography}{1}


\bibitem{ALR} A.I. Aptekarev, G. L\'opez Lagomasino, and I.A. Rocha,
{\em Ratio Asymptotic of Hermite-Pade orthogonal polynomials for
Nikishin systems,} Mat. Sb.  {\bf 196} (2005), 1089-1107.


\bibitem{AKLR} A.I. Aptekarev, V. Kalyagin, G. L\'opez Lagomasino, and I.A.
Rocha, {\em On the limit behavior of recurrence coefficients for
multiple orthogonal polynomials,} J. of Approx. Theory {\bf 139}
(2006), 346-370.



\bibitem{D} S. A. Denisov, {\em On Rakhmanov's theorem for Jacobi
matrices,} Proc. Amer. Math. Soc. {\bf 132} (2004), 847-852.


\bibitem{DeufHo} P. Deuflhard, A. Hofmann,
Numerical analysis in modern scientific computing. An introduction
(Second edition). Texts in Applied Mathematics, 43. Springer, New
York, 2003.



\bibitem{LL} A. L\'{o}pez Garc\'{\i}a and G. L\'{o}pez Lagomasino,
{\em Ratio asymptotic of Hermite-Pad\'{e} orthogonal polynomials
for Nikishin systems. II,}  Advances in Math. {\bf 218} (2008),
1081--1106.



\bibitem{N} P. Nevai, {\em Orthogonal Polynomials,} Memoires A.M.S.,
Vol. 213, Providence, R. I., 1979.


\bibitem{OrtRh}
J.M. Ortega, W. C. Rheinboldt, Iterative solutions of nonlinear
equations in several variables, Acad. Press, N.Y.--London, 1970.


\bibitem{Poi} H. Poincar\'e, Les M\'ethodes Nouvelles de la M\'ecanique
Celeste. Gauthier-Villars, Paris, 1892.



\bibitem{R} E. A. Rakhmanov, {\em On asymptotic properties of orthogonal
polynomials on the unit circle with weights not satisfying
Szeg\H{o}'s condition,}    Math. USSR Sb.  {\bf 58} (1987),
149-167.



\bibitem{ShV} B. Shapiro, A. Vainstein,
Counting real rational functions with all real critical values,
Moscow Math. J. {\bf 3} no. 2 (2003), 647--659.






\end{thebibliography}
\end{document}